\documentclass[reqno]{amsart}

\usepackage{amsmath, amssymb, amsthm}

\usepackage[pdftex, colorlinks=true, urlcolor=blue, linkcolor=blue, citecolor=blue, pdfstartview=FitH]{hyperref}

\newtheorem{theorem}{Theorem}

\newtheorem{proposition}{Proposition}

\newtheorem{corollary}{Corollary}

\newtheorem{remark}{Remark}

\everymath{\displaystyle}

\title{Higher order Poincare inequalities and Minkowski-type inequalities}

\author{Kwok-Kun Kwong}
\address{School of Mathematics and Applied Statistics, University of Wollongong, NSW 2522, Australia}
\email{kwongk@uow.edu.au}

\begin{document}

\maketitle
\begin{abstract}
We observe some higher order Poincare-type inequalities on a closed manifold, which is inspired by Hurwitz's proof of the Wirtinger's inequality using Fourier theory. We then give some geometric implication of these inequalities by applying them on the sphere. More specifically, by applying them to the support function of a convex hypersurface in the Euclidean space, we obtain some sharp Minkowski-type inequalities, such as a stability inequality for the classical Minkowski inequality and the Alexandrov-Fenchel inequality.
\end{abstract}

In a recent preprint \cite{kwong2020higher}, inspired by Hurwitz's proof (\cite{Hurwitz1902}, \cite{Chavel1978}) of the isoperimetric inequality, by extending the Wirtinger's inequality to its higher order analogue, Lee and the author of this paper were able to obtain sharp bounds for the isoperimetric deficit in terms of the curvature and its derivatives. The method is by Fourier analysis. Of course, the Wirtinger's inequality is the one-dimensional Poincare inequality on the circle and the theory of Fourier analysis can be extended to a large extent to the analysis of the spherical harmonics in higher dimensions. It is therefore natural to look for higher order Poincare inequalities by using spherical harmonics and apply them to obtain new geometric inequalities, which is the goal of this paper.

In general, it is well-known that on $\mathbb S^{d-1}$, if $F$ has mean zero, then we have the Poincare inequality $(d-1)\int_{\mathbb S^{d-1}} F^2\le \int_{\mathbb S^{d-1}} |\nabla F|^2$, which can be written as
$$-\int_{\mathbb S^{d-1}}F\Delta F-(d-1)\int_{\mathbb S^{d-1}}|F|^2\ge0. $$
This inequality alone can already be used to prove the remarkable Minkowski inequality (\cite[p. 1191]{Osserman1978}, \cite[p. 387]{schneider2014convex}) $\int_{\Sigma} H_{d-3} \le \frac{1}{\left| \mathbb S^{d-1}\right|}\left(\int_{\Sigma} H_{d-2}\right)^{2}$ for a convex hypersurface $\Sigma$ in $\mathbb R^d$. Here $H_k$ is the (normalized) $k$-th mean curvature and when $d=2$ this is the isoperimetric inequality. By a simple observation, we can extend the Poincare inequality to arbitrary high order. For example, again under the mean-zero condition, we have
\begin{equation}\label{intro ineq}
0 \le \int_{\mathbb S^{d-1}}\left(\Delta F+(d-1) F\right)^{2}-(d+1) \int_{\mathbb S^{d-1}}F\left(-\Delta F-(d-1)F \right).
\end{equation}
While this type of inequalities is perhaps known to experts, its geometric implication seems to be unexplored. For example, using \eqref{intro ineq}, we can obtain a stability result of the Minkowski inequality for convex hypersurfaces (Theorem \ref{thm2}):
\begin{align*}
\frac{1}{\left|\mathbb S^{d-1}\right|}\left(\int_{\Sigma} H_{d-2}\right)^{2}-\int_{\Sigma} H_{d-3} \le \frac{d-1}{d+1}\left[\int_{\Sigma} \frac{H_{d-2}^{2}}{H_{d-1}}-\frac{1}{\left|\mathbb S^{d-1}\right|}\left(\int_{\Sigma} H_{d-2}\right)^{2}\right].
\end{align*}
When $d=2$, this inequality is equivalent to
$ \frac{2}{\pi}\left(L^{2}-4 \pi A\right)\le \int_{ C } \frac{1}{\kappa} d s-2 A$ (\cite[Lemma 1.7]{LinTsai2012}), and thus establishes a connection between the isoperimetric inequality and the Ros inequality (\cite[Theorem 1]{Ros1987}).

The theory of convex hypersurfaces can be further extended to the theory of mixed volumes. While we mainly focus on obtaining results for one single convex body, it is possible to extend our analysis to mixed volumes of two convex bodies, as illustrated in Theorem \ref{mixed}. We refer to \cite{schneider2014convex} for an excellent reference on the theory of convex bodies and mixed volumes.

\textbf{Acknowledgement}: We would like to thank Hojoo Lee and Glen Wheeler for delightful discussion.

\section{A generalization of the Poincare inequality}
The analysis in this section can definitely be extended to elliptic operators with spectral properties analogous to the Laplacian, but for simplicity we treat the case of the Laplacian. Let $(M, g)$ be a closed orientable Riemannian manifold. Let $\Delta=\mathrm{div}(\nabla)$ be the Laplacian on $M$ (convention: $\Delta f=f''$ on $\mathbb S^1$).
Assume that the eigenvalues of $-\Delta$ on $M$ are (counting without multiplicities)
\begin{align*}
0=\lambda_0<\lambda_1<\lambda_2<\cdots\to\infty.
\end{align*}
Let $H_n=\{f\in C^\infty(M)|-\Delta f=\lambda_n f\}$ be the eigenspace with respect to $\lambda_n$.

Let $\langle \cdot, \cdot \rangle $ be the $L^2$ inner product on $(M, g)$ and $\|F\|^2=\langle F, F\rangle $. Denote the projection of $F$ onto $H_n$ by $F_n$, written as $F\sim\sum_{n=0}^{\infty}F_n$.

The following observation is quite straightforward, and is inspired by Wirtinger's proof of the Wirtinger inequality. It
provides a sharp higher order Poincare inequality if we have the knowledge of the first $m$ non-zero Laplacian eigenvalues.
\begin{proposition}
Let $m \ge l \ge 1$. Suppose $F\in C^\infty(M)$ and $F \sim \sum_{n=l}^{\infty} F_n$.
Then
\begin{align}\label{prod}
\left\langle\prod_{j=l}^{m}\left(-\Delta-\lambda_{j}\right) F, F\right\rangle \ge 0.
\end{align}
The equality holds if and only if
$F =\sum_{n=l}^m F_n$.
\end{proposition}
\begin{proof}
This is true because $-\Delta F\sim\sum_{n=l}^{\infty}\lambda_n F_n$ and so by the generalized Parseval's identity,
\begin{align*}
\left\langle\prod_{j=l}^{m}\left(-\Delta-\lambda_{j}\right) F, F\right\rangle= \sum_{n=m+1}^{\infty} \prod_{j=l}^{m}\left(\lambda_{n}-\lambda_{j}\right)\left\|F_{n}\right\|^{2} \ge 0.
\end{align*}
\end{proof}
For the purpose of deriving some sharp inequalities on convex bodies, we need the explicit expression of the coefficients.

Define the $(m-l+1)$-degree polynomial $C_{l, m}$ and the coefficients $c_{l, m, k}$ by
\begin{align*}
C_{l, m}(t):=\prod_{j=l}^{m}(t-\lambda_j)=\sum_{k=0}^{m-l+1} c_{l, m, k} t^{k}.
\end{align*}

\begin{proposition}\label{prop M}
Let $m \ge l \ge 1$. Suppose $F\in C^\infty(M)$ and $F \sim \sum_{n=l}^{\infty} F_n$.
Then we have
\begin{align}\label{ineq0}
\sum_{k=0}^{m-l+1} c_{l, m, k}\left\langle F, (-\Delta)^{k} F\right\rangle \ge 0.
\end{align}
The equality holds if and only if
$F =\sum_{n=l}^m F_n$.
\end{proposition}

\begin{corollary}
Let $m \ge 1 $. Suppose $F$ is a smooth function which has mean zero (i.e. $F\sim\sum_{n=1}^{\infty}F_n$), then
\begin{align*}
(-1)^{m-k}\sum_{k=0}^{m} \sigma_{m-k}(\Lambda)\left\langle F, (-\Delta)^{k} F\right\rangle \ge 0
\end{align*}
where $\Lambda=\left(\lambda_{1}, \cdots, \lambda_{m}\right)$ and $\sigma_{j}$ is the $j$-th elementary polynomial.
\end{corollary}

\subsection{Poincare-type inequalities on the sphere}

In this subsection we assume $M=\mathbb S^{d-1}$, the standard unit sphere in $\mathbb R^d$. We say $F$ has vanishing spherical harmonics up to order $k$ if $F_0=\cdots=F_k=0$ (\cite{Groemer1996}). It is well-known that $\lambda_{n}=n(n+d-2)$, so by Proposition \ref{prop M} we have the following inequalities.

\begin{enumerate}
\item
Let $m=1$ and $\int_{\mathbb S^{d-1}}F=0$. Then
$C_{1, 1}(t)= t-(d-1)$.
So $c_{1, 0}=-(d-1)$, $c_{1, 1}=1$. Therefore
\begin{align}\label{ineq0}
-\langle \Delta F+(d-1)F, F\rangle =
-(d-1)\langle F, F\rangle-\langle F, \Delta F\rangle \ge0.
\end{align}
This is the Poincare inequality.
\item
Assume $m=2$ and $\int_{\mathbb S^{d-1}}F=0$. We compute
\begin{align*}
C_{1, 2}(t)=(t-(d-1))(t-2d)=t^2-(3d-1)t+2d(d-1).
\end{align*}
So
\begin{align*}
2d(d-1)\langle F, F\rangle -(3d-1)\langle F, -\Delta F\rangle +\langle F, \Delta^2F\rangle\ge0.
\end{align*}
\item
Suppose $F$ has vanishing zeroth and first spherical harmonics, i.e.
$\displaystyle F\sim \sum_{n=2}^{\infty}F_n$, then \eqref{ineq0} gives $\displaystyle -\langle\Delta F, F\rangle \ge \lambda_2\langle F, F\rangle$. As $\lambda_2-\lambda_1=d+1$, this can be expressed as
\begin{align}\label{ineq0'}
-\langle F, \Delta F+(d-1) F\rangle \ge(d+1)\langle F, F\rangle.
\end{align}
\item
We have $ (-\Delta -\lambda_1) (-\Delta-\lambda_2)=[ -\Delta-(d-1) ]^2-(d+1) [-\Delta-(d-1) ] $,
So if $\int_{\mathbb S^{d-1}}F=0$, then by \eqref{prod},
\begin{align}\label{eg4}
0 \le \|\Delta F+(d-1) F\|^{2}-(d+1)\langle F, [-\Delta-(d-1) ] F\rangle.
\end{align}

\item
Similarly,
\begin{align*}
(-\Delta-\lambda_2) (-\Delta-\lambda_3)=B^2-(3d+5)B+2(d+1)(d+2)
\end{align*}
where $B=-\Delta-(d-1)$.
So if $F\sim \sum_{n=2}^{\infty} F_{n}$, then by \eqref{prod},
\begin{align}\label{eg5}
0\le\left\|\left[ \Delta +(d-1) \right] F\right\|^{2}-(3d+5)\left\langle F, \left[ -\Delta -(d-1) \right]F\right\rangle+2(d+1) (d+2) \left\|F\right\|^{2}.
\end{align}
\end{enumerate}
\section{Geometry of convex bodies}
From now on, unless otherwise stated, $\langle \cdot, \cdot\rangle $ and $\|\cdot\|$ denote the inner product and the norm respectively on $\mathbb S^{d-1}$.
\subsection{Single convex body}
We collect some formulas for convex bodies here. Let $K\subset \mathbb R^d$ be a smooth convex body. The support function $h=h_K:\mathbb S^{d-1}\to \mathbb R$ is defined by $h_K(\theta)=\sup_{x\in K} (x\cdot \theta) $.
The Steiner point $z(K)$ of $K$ is defined by (\cite[Sec. 2.6]{Groemer1996})
\begin{align*}
z(K)=\frac{1}{|\mathbb B^d|} \int_{\mathbb S^{d-1}} h_{K}(\theta) \, \theta\, d \theta
\end{align*}
where $|\mathbb B^d|$ is the volume of the unit ball in $\mathbb R^d$ and $d\theta$ denotes the area element of $\mathbb S^{d-1}$.

The ball whose center is the Steiner point and whose diameter is the mean width $\overline w(K):=\frac{2}{|\mathbb S^{d-1}|}\int_{\mathbb S^{d-1}}h_K(\theta) d\theta$ of $K$ will be called the Steiner ball of $K$ and is denoted by $B(K)$.
The support function of the Steiner ball is
\begin{align*}
h_{B(K)}(\theta)=\frac{1}{2} \overline {w}(K)+z(K) \cdot \theta.
\end{align*}
Given $h_K\sim\sum_{n=0}^{\infty}F_n$, we have the formula for the $\delta_2$ distance between $K$ and $B(K)$ \cite[(5.2.3)]{Groemer1996}:
\begin{align}\label{delta2}
\delta_{2}\left(K, B(K)\right)^{2}:=\int_{\mathbb S^{d-1}}(h_K-h_{B(K)})^2\, d\theta=\sum_{n=2}^{\infty}\left\|F_n\right\|^{2}.
\end{align}
Let $\Sigma$ be the boundary of $K$, which is a closed convex hypersurface in $\mathbb R^{d}$. Let $\{e_i\}_{i=1}^{d-1}$ be a local orthonormal basis at $\theta\in \mathbb S^{d-1}$, $\nabla $ be the connection on $\mathbb S^{d-1}$ and $X$ be the position vector, which can be regarded as the inverse Gauss map $X: \mathbb S^{d-1}\to \Sigma$.
By identifying $T\Sigma$ with $T\mathbb S^{d-1}$ by parallel translation (we do not identify $e_i$ with $DX(e_i)$),
\begin{align*}
D X\left(e_{i}\right) =h_{i} \theta+h e_{i}+h_{i j} e_{j}-h_{i} \theta
=h e_{i}+\nabla_{i j} h \;e_{j},
\end{align*}

Let $\nu$ be the unit outward normal to $\Sigma$.
The second fundamental form is represented by the matrix $A=(A_{ij})$ (this is not the same as $A(e_i, e_j)$), where
\begin{align*}
A_{ij}:=\langle D_{DX(e_i)}(\nu), DX(e_j)\rangle_\Sigma=\langle D_{e_i}(\nu\circ X), DX(e_j)\rangle_\Sigma
=&\langle D_{e_i}\theta, DX(e_j)\rangle_{\mathbb S^{d-1}}\\
=&\langle e_i, DX(e_j)\rangle_{\mathbb S^{d-1}}\\
=&h \delta_{ij}+\nabla_{ij} h.
\end{align*}
The first fundamental form of $\Sigma$ is given by
\begin{align*}
\left(g_{i j}\right):=\langle DX(e_i), DX(e_j)\rangle
=A^{2}.
\end{align*}
The shape operator $S$ is defined by $\langle S(Y), Z\rangle_\Sigma=\langle D_Y\nu, Z\rangle_\Sigma $, and its matrix representation is given by
$\left(g^{i j}\right)\left(A_{ jk}\right)=A^{-1}$.

We can without loss of generality assume that $\nabla^2 h$ is diagonal at a given point, from this it is easy to derive the following formulas
\begin{equation}\label{conv formulas}
\begin{split}
\begin{cases}
&H_{d-2}=\frac{1}{d-1} \sigma_{d-2}\left(A^{-1}\right)
=\frac{ \textrm{tr } A }{(d-1)\det A}
=\frac{(d-1)h+\Delta h}{(d-1)\det A} \\
&\frac{1}{H_{d-1}}=\det A \\
&\frac{H_{d-2}}{H_{d-1}}=h+\frac{1}{d-1}\Delta h \\
&d\theta=H_{d-1}dS.
\end{cases}
\end{split}
\end{equation}

It was proved by Steiner that for a convex body $K$, we have the following Steiner expansion formula \cite[4.2]{schneider2014convex}
\begin{align*}
\left| K+t \mathbb B^d \right|=\sum_{i=0}^{d}\binom{d}{i} W_{i}(K) t^{i}
\end{align*}
where $|\cdot|$ stands for the $d$-dimensional volume and $W_i(K)$ is a constant, called the $i$-th quermassintegral of $K$ (see Subsection \ref{mixed} for an equivalent definition). It can be shown that if $K$ is $C^2$, then $W_{d-i}(K)=\frac{1}{d} \int_{\Sigma} H_{d-1-i}dS$ (cf. \cite[4.2]{schneider2014convex}).

The mean width $\overline w$ is given by (\cite[p. 46]{Groemer1996})
\begin{align}\label{H d-2}
\frac{|\mathbb S^{d-1}|}{2}\overline w(K)=d \cdot W_{d-1}(K)=\int_\Sigma H_{d-2}dS=\int_{\Sigma} H_{d-1} \langle X, \nu\rangle dS=
\int_{\mathbb S^{d-1}}h \, d\theta.
\end{align}
Similarly,
\begin{equation} \label{H d-3}
\begin{split}
d\cdot W_{d-2}(K)=\int_\Sigma H_{d-3} =\int_\Sigma H_{d-2}\langle X, \nu\rangle dS
=&\int_{\mathbb S^{d-1}} \frac{h\;\mathrm{tr }A}{(d-1)\det A}\cdot \det A\;d\theta\\
=&\frac{1}{d-1}\int_{\mathbb S^{d-1}} ((d-1)h^2+h\Delta h) d\theta\\
=&\frac{1}{d-1}\langle h, \Delta h+(d-1)h\rangle.
\end{split}
\end{equation}
\subsection{Mixed volume of multiple convex bodies}\label{mixed}
The theory of one convex body can be generalized to the theory of mixed volumes. Let $K_{1}, K_{2}, \ldots, K_{d}$ be convex bodies in $\mathbb R^d$ and consider the function
\begin{align*}
f\left(t_{1}, \ldots, t_{d}\right)=\left| t_{1} K_{1}+\cdots+t_{d} K_{d} \right|, \quad t_{i} \ge 0.
\end{align*}
One can show that $f$ is a homogeneous polynomial of degree $d$, therefore it can be written as (\cite[Sec. 5.1]{schneider2014convex})
\begin{align*}
f\left(t_{1}, \ldots, t_{d}\right)=\sum_{j_{1}, \ldots, j_{d}=1}^{d} V\left(K_{j_{1}}, \ldots, K_{j_d}\right) t_{j_{1}} \cdots t_{j_d}
\end{align*}
where the functions $V$ are symmetric. The coefficient $V\left(K_1, \cdots, K_d\right)$ is called the mixed volume of $K_1, \cdots, K_d$.
We define the mixed volume $V(K, L):=V(K, L, \overbrace{\mathbb B^d, \cdots, \mathbb B^{d}}^{d-2})$ and the quermassintegral
$W_{i}(K):=V (\overbrace{K, \cdots, K}^{d-i}, \overbrace{\mathbb B^d, \cdots, \mathbb B^{d}}^{i})$. It can be shown that (\cite[Prop. 5.1.3 ]{Groemer1996})
\begin{align}\label{v kl}
V(K, L)=\frac{1}{d(d-1)}\left\langle h_{K}, \Delta h_{L}+(d-1) h_{L}\right\rangle=\frac{1}{d(d-1)}\left\langle h_{L}, \Delta h_{K}+(d-1) h_{K}\right\rangle.
\end{align}
This generalizes \eqref{H d-3}.

\section{Applications to Minkowski-type inequalities}
In this subsection, we are going to apply the inequalities we have obtained on the sphere to derive new sharp Minkowski-type inequalities.
These inequalities will be proved in the smooth case, although by approximation they also apply to general convex bodies.
The simplest case of Proposition \ref{prop M} (Poincare inequality) gives the following well-known Minkowski inequality (\cite[Theorem 5.2.1]{Groemer1996}) (at least without the $\delta_2$ term).
To illustrate its relation with Theorem \ref{thm2} and Theorem \ref{thm3}, we sketch its proof here.
\begin{theorem}
[Minkowski inequality]\label{thm1}
If $\Sigma$ is a closed convex hypersurface in $\mathbb R^{d}$, then
\begin{align*}
\int_\Sigma H_{d-3}+ \frac{ d+1 }{d-1}\delta_2(K, B(K))^2 \le \frac{1}{|\mathbb S^{d-1}|}\left(\int_\Sigma H_{d-2}\right)^2.
\end{align*}
\end{theorem}
\begin{proof}
Let $F=h-\frac{1}{|\mathbb S^{d-1}|}\int_{\mathbb S^{d-1}}h $. Then $F$ has vanishing zeroth and first spherical harmonics and indeed $F=h-h_{B(K)}$ as $\frac{1}{|\mathbb S^{d-1}|}\int_{\mathbb S^{d-1}}h=\frac{1}{2}\overline w(K)$ (\cite[(2.6.3)]{Groemer1996}).
By \eqref{delta2} and inequality \eqref{ineq0'},
\begin{equation*}\label{mink}
\begin{split}
(d+1)\delta_2(K, B(K))^2=
(d+1)\|F\|^2
\le&\langle F, -\Delta F-(d-1)F\rangle\\
=&-\left\langle h, \Delta h\right\rangle-(d-1)\langle h, h\rangle +\frac{d-1}{|\mathbb S^{d-1}|} \|h\|^2 \\
=&\frac{d-1}{|\mathbb S^{d-1}|}\left(\int_\Sigma H_{d-2}\right)^2-(d-1)\int_\Sigma H_{d-3}.
\end{split}
\end{equation*}
\end{proof}

\begin{remark}
This result is equivalent to $W_{d-1}(K)^{2}-W_d(K) W_{d-2}(K) \ge \kappa_{d} \frac{d+1}{d(d-1)} \delta_{2}\left(K, B(K)\right)^{2}$, as $W_d(K)=\kappa_d=|\mathbb B^d|$.
As illusioned in \cite[Theorem 5.2.2]{Groemer1996}, although $W_{d-2}$ and $W_{d-1}$ seem special, this result can be combined with the Minkowski inequality $W_{k}(K)^{2} \ge W_{k+1}(K) W_{k-1}(K)$ to obtain results relating $W_i$ and $W_j$ for general $i<j$.

Also, if we use the inequality \eqref{ineq0} instead of \eqref{ineq0'}, we get the classical Minkowski inequality
\begin{align*}
\int_{\Sigma} H_{d-3} \le \frac{1}{\left| \mathbb S^{d-1}\right|}\left(\int_{\Sigma} H_{d-2}\right)^{2}
\end{align*}
instead.
\end{remark}

The following result gives an upper bound for the deficit in the Minkowski formula. In the case where $d=2$ and with the convention that $H_{-1}=X\cdot \nu$, this is the inequality $L^{2}-4 \pi A \le \frac{2 \pi}{3}\left(\int_{C} \frac{1}{\kappa} d s-\frac{L^{2}}{2 \pi}\right)$, which is the Lin-Tsai inequality \cite[Lemma 1.7]{LinTsai2012}.
\begin{theorem}\label{thm2}
If $\Sigma$ is a closed convex hypersurface in $\mathbb R^{d}$, then
\begin{align*}
\frac{1}{\left| \mathbb S^{d-1}\right|}\left(\int_{\Sigma} H_{d-2}\right)^{2}-
\int_{\Sigma} H_{d-3}
\le \frac{d-1}{ d+1 } \left[ \int_{\Sigma} \frac{{H_{d -2}}^{2}}{H_{d-1}} -\frac{1}{ |\mathbb S^{d-1}|} \left(\int_{\Sigma} H_{d-2}\right)^{2}\right].
\end{align*}
\end{theorem}

\begin{proof}

Let $\displaystyle F=h - \overline h $, where $ \displaystyle \overline h =\frac{1}{\left| \mathbb S^{d-1}\right|} \int_{\mathbb S^{d-1}} hd\theta=\frac{1}{\left|\mathbb S^{d-1}\right|} \int_\Sigma H_{d-2}$ by \eqref{H d-2}.

By \eqref{eg4},
\begin{equation} \label{thm 2 ineq}
0 \le \|(\Delta+(d-1)) F\|^{2}+(d+1)\langle F, [\Delta+(d-1)] F\rangle.
\end{equation}

By \eqref{conv formulas},
\begin{equation}\label{B2}
\begin{split}
\frac{1}{(d-1)^2}\|(\Delta+(d-1)) F\|^{2}
=&\int_{\mathbb S^{d-1}}\left(\frac{\Delta h}{d-1}+h-\overline h \right)^{2} \\
=&\int_{\mathbb S^{d-1}}\left(\frac{\Delta h}{d-1}+h\right)^{2}-\frac{1}{| \mathbb S^{d-1}|} \left(\int_{\mathbb S^{d-1}} h\right)^{2} \\
=&\int_{\Sigma} \frac{{H_{d -2}}^{2}}{H_{d-1}} -\frac{1}{\left|\mathbb S^{d-1}\right|}\left(\int_{\Sigma} H_{d-2}\right)^{2}.
\end{split}
\end{equation}

On the other hand, by \eqref{H d-3},
\begin{equation}\label{B1}
\begin{split}
\langle F, [\Delta+(d-1)] F\rangle
=&\int_{\mathbb S^{d-1}}\left[(d-1)(h-\overline h)^{2}+h \Delta h\right]\\
=&\int_{\mathbb S^{d-1}}\left[(d-1) h^{2}+h \Delta h\right]-\frac{d-1}{\left|\mathbb S^{d-1}\right|}\left(\int_\Sigma H_{d -2}\right)^{2}\\
=&(d-1) \int_{\Sigma} H_{d -3}-\frac{d-1}{\left|\mathbb S^{d-1}\right|}\left(\int_\Sigma H_{d-2}\right)^{2}.
\end{split}
\end{equation}

So \eqref{thm 2 ineq} becomes
\begin{align*}
0\le(d-1)^{2}\left(\int_{\Sigma} \frac{{H_{d -2}}^{2}}{H_{d-1}} -\frac{1}{ |\mathbb S^{d-1}|} \left(\int_{\Sigma} H_{d-2}\right)^{2}\right)+
(d+1)\left[ (d-1) \int_{\Sigma} H_{d-3}-\frac{d-1}{\left| \mathbb S^{d-1}\right|}\left(\int_{\Sigma} H_{d-2}\right)^{2}\right].
\end{align*}
\end{proof}

The following result gives a reverse inequality to Theorem \ref{thm1} and improves Theorem \ref{thm2} at the same time.
\begin{theorem}\label{thm3}
If $\Sigma=\partial K$ is a closed convex hypersurface in $\mathbb R^{d}$, then

\begin{align*}
&\frac{ d-1 }{d+1}\left[\int_{\Sigma} \frac{{H_{d-2}}^{2}}{H_{d-1}}-\frac{1}{\left| \mathbb S^{d-1}\right|}\left(\int_{\Sigma} H_{d-2}\right)^{2}\right]-\left[\frac{1}{\left| \mathbb S^{d-1}\right|}\left(\int_{\Sigma} H_{d-2}\right)^{2}-\int_{\Sigma} H_{d-3}\right]\\
\ge&\frac{2(d+2)}{d+1}
\left[ \frac{1}{\left| \mathbb S^{d-1}\right|}\left(\int_{\Sigma} H_{d-2}\right)^{2}-\int_{\Sigma} H_{d-3}
-\frac{ d+1 }{d-1}\delta_{2}\left(K, B(K)\right)^{2} \right].
\end{align*}
\end{theorem}

\begin{proof}

We use the same notation and assumption in Theorem \ref{thm1}.
We rewrite \eqref{eg5} as
\begin{equation}
\begin{split}
\label{thm 3 ineq}
&\left\|\Delta F+(d-1) F\right\|^{2}-(d+1)\left\langle F, [ -\Delta-(d-1) ] F\right\rangle\\
&-2(d+2)\left[\langle F, (-\Delta-(d-1)) F\rangle-(d+1)\|F\|^2\right]\ge0.
\end{split}
\end{equation}

As in Theorem \ref{thm2},
\begin{align*}
&\|\Delta F+(d-1) F\|^{2}-(d+1)\langle F, (-\Delta-(d-1)) F\rangle\\
=&(d-1)^2
\left[ \int_{\Sigma} \frac{{H_{d-2}}^{2}}{H_{d-1}}-\frac{1}{\left| \mathbb S^{d-1}\right|}\left(\int_{\Sigma} H_{d-2}\right)^{2}\right]-(d+1)(d-1)
\left[\frac{1}{|\mathbb S^{d-1}|}\left(\int_{\Sigma}H_{d-2}\right)^{2}-\int_{\Sigma} H_{d-3}\right].
\end{align*}
As in Theorem \ref{thm1},
\begin{align*}
&\langle F, (-\Delta-(d-1)) F\rangle+(d+1)\|F\|^{2}\\
=&(d-1)\left(\frac{1}{\left|\mathbb S^{d-1}\right|}\left(\int_{\Sigma} H_{d-2} \right)^{2} - \int_{\Sigma} H_{d-3}\right)- (d+1)\delta_{2}\left(K, B(K)\right)^{2}.
\end{align*}
Putting these into \eqref{thm 3 ineq} gives the result.
\end{proof}

In \cite[Theorem 5]{kwong2020higher}, the authors proved some sharp upper and lower bounds for the isoperimetric deficit of a convex curve which involves arbitrarily high order of the curvature. For example, it was proved that $0 \le\left(\int_{ C } \frac{1}{\kappa} d s-\frac{L^{2}}{2 \pi}\right)-\frac{3}{2 \pi}\left(L^{2}-4 \pi A\right) \le \frac{1}{12}\left[\int_{ C } \frac{1}{\kappa^{5}}\left(\frac{d \kappa}{d s}\right)^{2} d s-\frac{6}{\pi}\left(L^{2}-4 \pi A\right)\right]$. The following result is the extension to higher dimensions.

\begin{theorem}
Let $m\in \mathbb N$.
If $\Sigma$ is a closed convex hypersurface in $\mathbb R^{d}, $ then under the inverse Gauss map representation,
\begin{align*}
0\le&(-1)^{m} \sum_{i=1}^{m-2} c_{i}\left\langle\Delta^{i} \rho, p\right\rangle \\
&+(-1)^{m} \frac{(d-1)^{2}(m-1)(m-1) !(m+d-2) !}{d !}\left[\int_{\Sigma} \frac{H_{d-2}^{2}}{H_{d-1}}-\frac{1}{\left| \mathbb S^{d-1}\right|}\left(\int_{\Sigma} H_{d-2}\right)^{2}\right] \\
&+(-1)^{m-1} \frac{(d-1)(m-1) !(m+d-1) !}{d !}\left[\frac{1}{\left| \mathbb S^{d-1}\right|}\left(\int_{\Sigma} H_{d-2}\right)^{2}-\int_{\Sigma} H_{d-3}\right].
\end{align*}

Here $\rho=\frac{H_{d-2}}{H_{d-1}}$, $\lambda_{n}=n(n+d-2)$, $\Delta$ is the spherical Laplacian,
$c_i=\sum_{i=1}^{m-2} \sum_{l=i}^{m-2} \sigma_{m-2-l}(\Gamma)\binom{l}{i} \lambda_1^{l-i+2} $, and $\Gamma=\left\{0=\gamma_1, \gamma_{2}, \cdots, \gamma_{m}\right\}$ where $\gamma_{n}=\lambda_{n}-\lambda_{1}=(n-1)(n+d-1)$.
\end{theorem}
\begin{proof}
Recall that on $\mathbb S^{d-1}$, $\lambda_n=n(n+d-2)$.
Let $B=-\Delta -\lambda_1$ and $\gamma_n=\lambda_n-\lambda_1=(n-1)(n+d-1)$. Then $\langle PF, F\rangle \ge 0$ whenever $\int_{\mathbb S^{d-1}}F=0$, where $P=\prod_{n=1}^m (B-\gamma_n) $.
Expanding $P$ gives
\begin{align}\label{P}
P
=&\sum_{l=1}^{m-2}(-1)^{m-l} \sigma_{m-2-l}(\Gamma) B^{l+2}+(-1)^{m-2} \sigma_{m-2}(\Gamma) B^{2}+(-1)^{m-1} \sigma_{m-1}(\Gamma) B.
\end{align}

Let $F=h-\overline h$ and $\rho=\frac{H_{d-2}}{H_{d-1}}$, then $BF=B(h-\overline {h})=-\lambda_{1} \rho+\lambda_{1} \overline {h}$.
By \eqref{B2}, we have
\begin{align*}
\langle B^{l+2}F, F\rangle
=&{\lambda_1}^2\left\langle B^{l}(-\rho+\overline {h}), -\rho+\overline {h}\right\rangle \\
=& {\lambda_1}^2 \sum_{i=1}^{l} {\binom{l}{i}}(-\lambda_1)^{l-i}\left\langle(-\Delta)^{i}(-\rho + \overline {h}), -\rho+\overline {h}\right\rangle+(-\lambda_1)^{l+2}\|-\rho+\overline h\|^2\\
=&\sum_{i=1}^{l} {\binom{l}{i}}(-\lambda_1)^{l-i+2}\left\langle(-\Delta)^{i} \rho, \rho \right\rangle+(-\lambda_1)^{l+2}
\left[\int_{\Sigma} \frac{{H_{d-2}}^{2}}{H_{d-1}}-\frac{1}{\left|\mathbb S^{d-1}\right|}\left(\int_{\Sigma} H_{d-2}\right)^{2}\right].
\end{align*}

So by \eqref{P}, \eqref{B2} and \eqref{B1},
\begin{equation}\label{PF}
\begin{split}
0\le&\langle PF, F\rangle\\
=& \sum_{l=1}^{m-2}(-1)^{m-l} \sigma_{m-2-l}(\Gamma) \left\{ \sum_{i=1}^{l}\binom{l}{i}\left(-\lambda_{1}\right)^{l-i+2}\left\langle(-\Delta)^{i} \rho, \rho\right\rangle\right. \\
&\left. +\left(-\lambda_{1}\right)^{l+2}\left[\int_{\Sigma} \frac{{H_{d-2}}^{2}}{H_{d-1}}-\frac{1}{\left| \mathbb S^{d-1}\right|}\left(\int_{\Sigma} H_{d-2}\right)^{2}\right]\right\}\\
&+(-1)^{m} \sigma_{m-2}(\Gamma) \lambda_1^2\left[\int_{\Sigma} \frac{{H_{d-2}}^{2}}{H_{d-1}}-\frac{1}{\left| \mathbb S^{d-1}\right|}\left(\int_{\Sigma} H_{d-2}\right)^{2}\right]\\
&+(-1)^{m-1} \sigma_{m-1}(\Gamma) \left[\frac{d-1}{\left| \mathbb S^{d-1}\right|}\left(\int_{\Sigma} H_{d-2}\right)^{2}-(d-1) \int_{\Sigma} H_{d-3}\right].
\end{split}
\end{equation}
We compute
\begin{equation}\label{ci}
\begin{split}
&\sum_{l=1}^{m-2} \sum_{i=1}^{l}(-1)^{m-l} \sigma_{m-2-l}\left(\Gamma\right)\binom{l}{i}(-\lambda_1)^{l-i+2}(-\Delta)^{i}\\
=&\sum_{i=1}^{m-2} \sum_{l=i}^{m-2}(-1)^{m-l} \sigma_{m-2-l}\left(\Gamma\right)\binom{l}{i}(-\lambda_1)^{l-i+2}(-\Delta)^{i}\\
=&(-1)^m\sum_{i=1}^{m-2}\left[\sum_{l=i}^{m-2} \sigma_{m-2-l}(\Gamma)\binom{l}{i} \lambda_1^{l+2}\right]\left(\frac{\Delta}{\lambda_1}\right)^{i}.
\end{split}
\end{equation}
We compute the coefficient of $\left[\int_{\Sigma} \frac{H_{d-2}^{2}}{H_{d-1}}-\frac{1}{\left|\mathbb S^{d-1}\right|}\left(\int_{\Sigma} H_{d-2}\right)^{2}\right]$ in \eqref{PF} to be
\begin{equation}\label{coeff1}
\begin{split}
(-1)^{m}\left[\sum_{l=0}^{m-2} \sigma_{m-2-l}(\Gamma) \lambda_{1}^{l+2}\right]
=&(-1)^{m}\left[\prod_{j=1}^{m} \lambda_{j}-\lambda_{1} \prod_{j=2}^{m} \gamma_{j}\right]\\
=&(-1)^{m} \frac{(d-1)^{2}(m-1)(m-1) !(m+d-2) !}{d !}.
\end{split}
\end{equation}
We compute the coefficient of $\left[\frac{1}{\left| \mathbb S^{d-1}\right|}\left(\int_{\Sigma} H_{d-2}\right)^{2}- \int_{\Sigma} H_{d-3}\right]$ in \eqref{PF} to be
\begin{align}\label{coeff2}
(-1)^{m-1} \lambda_{1} \sigma_{m-1}(\Gamma)=(-1)^{m-1} \lambda_{1} \prod_{j=2}^{m} \gamma_{j}=(-1)^{m-1} \frac{(d-1)(m-1) !(m+d-1) !}{d !}.
\end{align}
Putting \eqref{ci}, \eqref{coeff1}, \eqref{coeff2} into \eqref{PF}, we get the result.
\end{proof}

The above analysis can also be carried out in the setting of mixed volumes.
It is well-known that for two convex bodies $K$ and $L$, we have the Aleksandrov-Fenchel inequality (\cite[Theorem 5.2.3]{Groemer1996}, \cite[Theorem 7.3.2]{schneider2014convex})
$$V(K, L)^2-W_{d-2}(K)W_{d-2}(L)\ge 0. $$
When $d=2$ and $L=\mathbb B^2$, this reduces to the isoperimetric inequality for the convex curve $\partial K$.
We now prove a reverse Aleksandrov-Fenchel inequality.

\begin{theorem}\label{mixed}
If $K$ and $L$ are convex bodies in $\mathbb R^{d}$, translated so that they have the same Steiner point,
then
\begin{align*}
&V(K, L)^{2}-W_{d-2}(K) W_{d-2}(L)\\
\le&\left(\frac{ \overline w (K)}{\overline w(L)}W_{d-2}(L)-V(K, L)\right)^{2}\\
&+W_{d-2}(L)
\left[ \frac{ d-1 }{(3 d+5)d} \int_{\mathbb S^{d-1}}\left(\frac{H_{d-2}(K)}{H_{d-1}(K)}- \frac{\overline {w}(K) H_{d-2}(L)}{ \overline {w}(L) H_{d-1}(L)}\right)^{2} d \theta+\frac{2(d+1)(d+2)}{d(3 d+5)(d-1)}\delta_{2}(K, L)^{2} \right].
\end{align*}
\end{theorem}

\begin{proof}
We can assume that $z(K)=z(L)=0$ (\cite[Sec. 1.7]{schneider2014convex}). We first prove the case where $\overline w(K)=\overline w(L)$.
Let $\displaystyle h_K \sim \sum_{n=0}^{\infty} \textsf{K}_n$, $h_L\sim \sum_{n=0}^{\infty} \textsf{L}_{n}$.
By the assumption,
$\textsf{K}_{1}=\textsf{L}_{1}=0$, $\textsf{K}_{0}=\textsf{L}_{0}$.
In particular, $F:=h_K-h_L$ has vanishing zeroth and first spherical harmonics. We have
\begin{equation} \label{p0}
\begin{split}
&V(K, L)^{2}-W_{d-2}(K) W_{d-2}(L)\\
=&\left(W_{d-2}(L)-V(K, L)\right)^{2}-W_{d-2}(L)\left[ W_{d-2}(K)+W_{d-2}(L)-2 V(K, L) \right].
\end{split}
\end{equation}

We consider the term $W_{d-2}(K)+W_{d-2}(L)-2 V(K, L)$.
By \eqref{H d-3} and \eqref{v kl},
\begin{equation}\label{p1}
\begin{split}
d(d-1)\left[W_{d-2}(K)+W_{d-2}(L)-2 V(K, L)\right]
=&\left\langle h_{K}-h_{L}, [\Delta+(d-1)]\left(h_{K}-h_{L}\right)\right\rangle\\
=&\left\langle F, \Delta F+(d-1) F\right\rangle.
\end{split}
\end{equation}

The inequality \eqref{eg5} gives
\begin{align*}
\left\langle F, \left[\Delta +(d-1) \right]F\right\rangle\ge-\frac{1}{3 d+5} \left\|[\Delta+(d-1)] F\right\|^{2}-\frac{2(d+1)(d+2)}{3 d+5}\|F\|^{2}
\end{align*}
By \eqref{delta2},
$ \|F\|^2=\delta_2(K, L)^2$ and by \eqref{conv formulas},
\begin{align*}
\left\| \Delta F +(d-1)F \right\|^{2}=(d-1)^2\int_{ \mathbb S^{d-1}}\left(\frac{H_{d-2}(K)}{H_{d-1}(K)}-\frac{H_{d-2}(L)}{H_{d-1}(L)}\right)^{2} d \theta.
\end{align*}
Put these into \eqref{p1} and then \eqref{p0},
we obtain
\begin{align*}
&V(K, L)^{2}-W_{d-2}(K) W_{d-2}(L)\\
\le&\left(W_{d-2}(L)-V(K, L)\right)^{2}\\
&+W_{d-2}(L)
\left[ \frac{ d-1 }{(3 d+5)d} \int_{\mathbb S^{d-1}}\left(\frac{H_{d-2}(K)}{H_{d-1}(K)}-\frac{H_{d-2}(L)}{H_{d-1}(L)}\right)^{2} d \theta+\frac{2(d+1)(d+2)}{d(3 d+5)(d-1)}\delta_{2}(K, L)^{2} \right].
\end{align*}
The general case can be deduced from above by replacing $K$ with $\frac{\overline w(L)}{\overline w(K)}K$.
\end{proof}


\begin{thebibliography}{99}

\bibitem{Chavel1978} I. Chavel, \emph{On A. Hurwitz' method in isoperimetric inequalities}, Proc. Amer. Math. Soc. 71 (1978), no. 2, 275--279.

\bibitem{Groemer1996} H. Groemer, \emph{Geometric applications of Fourier series and spherical harmonics}, Cambridge University Press, 1996.

\bibitem{Hurwitz1902} A. Hurwitz, \emph{Sur quelques applications g\'{e}om\'{e}triques des s\'{e}ries de Fourier}, Ann. Sci. \'{E}cole Norm. Sup. (3) 19 (1902), 357--408.

\bibitem{kwong2020higher}
K-K. Kwong and H. Lee,
\newblock Higher order wirtinger-type inequalities and sharp bounds for the isoperimetric deficit.
\newblock {\em arXiv preprint arXiv:2008.07242}, 2020. To appear in Proc. of the AMS.

\bibitem{LinTsai2012} Y.-C. Lin and D.-H. Tsai, \newblock \emph{Application of Andrews and Green-Osher inequalities to nonlocal flow of convex plane curves}, \newblock J. Evol. Equ. 12 (2012), no. 4, 833--854.

\bibitem{Osserman1978} R. Osserman, \emph{The isoperimetric inequality}, Bull. Amer. Math. Soc. 84 (1978), no. 6, 1182--1238.

\bibitem{Ros1987} A. Ros, \textit{Compact hypersurfaces with constant higher order mean curvatures}, Rev. Mat. Iberoamericana \textbf{3} (1987), 447--453.

\bibitem{schneider2014convex} R. Schneider,
\newblock {\em Convex bodies: the Brunn--Minkowski theory}.
\newblock Number 151. Cambridge university press, 2014.
\end{thebibliography}
\end{document}